\documentclass[preprint,12pt]{elsarticle}
\usepackage[a4paper, total={6in, 8in}, margin=0.95in]{geometry}
\usepackage{amsmath,amsfonts,amssymb,amsthm,epsfig,epstopdf,url,array}
\usepackage[retainorgcmds]{IEEEtrantools}
\usepackage{graphicx}
\usepackage{xcolor}
\usepackage{transparent}
\usepackage{latexsym}
\usepackage{float}
\usepackage[font=small,labelfont=bf]{caption}
\usepackage{subcaption}
\usepackage{lineno}

\theoremstyle{plain}
\newtheorem{thm}{Theorem}[section]
\newtheorem{lem}[thm]{Lemma}

\newtheorem{cor}{Corollary}

\theoremstyle{plain}
\newtheorem{defn}{Definition}[section]

\theoremstyle{definition}
\newtheorem*{rem}{Remark}

\newcommand{\cref}[1]{{\it Chapter~\ref{#1}}}

\numberwithin{equation}{section}
\journal{Journal}
\begin{document}

\begin{frontmatter}
\title{Stability and decay estimates for the marine riser equation in the presence of time dependent data.}%\tnoteref{t1}
%\tnotetext[t1]{PhD thesis supported by MEDASTAR office/ University of Oviedo/ Spain.}
\author{Waleed S. Khedr}%\corref{cor1}}
\ead{waleedshawki@yahoo.com}
%\cortext[cor1]{Corresponding author.}
%\author{Sergey I. Shmarev}
%\ead{shmarev@uniovi.es}
\address{waleedshawki@yahoo.com}
\begin{abstract}
In this article we investigate the dynamics of the initial-boundary value problem for the nonlinear marine riser equation in the presence of time dependent boundary conditions at the top end and a time dependent coefficient of the nonlinear drag force. We introduce sufficient conditions on these functions to maintain the structural stability of the system. We deduce their maximum rates of growth to guarantee that the zero solution is globally asymptotically stable.
\end{abstract}
\end{frontmatter}
%000000000000000000000000000
\section{Introduction}
The marine riser equation represents the balance of forces exerted on an elastic pipe. While some of these forces are causing displacements which we interpret as deflection in the body of the pipe, other forces are exerted towards the restoration of the pipe to its natural vertical orientation which we describe it mathematically as the zero solution. In mathematical studies of this equation, it is common to assume non zero initial data that represent a deflected pipe, and then we study the conditions under which the pipe can be restored to its vertical position (the zero solution). Some of these forces are exerted naturally due to the rigidity of the pipe, the effects of the fluid flowing inside or the external effects as waves and currents in the sea; others are added to play the role of the dissipative component of the system and these can be proportional to the amount of displacement (as in the simple case of using a spring) or they can be proportional to the rate of deflection so that they can react dynamically to any forces deforming the pipe; this is what we commonly refer to as the nonlinear drag force. More details about modelling and engineering aspects of the problem can be found in a number of publications and books, for example see \cite{bernitsas,jjh,effect2014}.\\

Investigating the global asymptotic stability of zero solutions of higher order nonlinear wave equations was carried out by P. Marcati in \cite{marcati84}. In \cite{kohl93} M. K\"ohl considered a one dimensional marine riser equation with constant coefficients, quadratic drag force, and fixed boundaries and he proved that the zero solution is stable. A more general model with higher order of nonlinearity and spatial dependent coefficient of effective tension was considered in \cite{kalantarov97}, where the authors managed to prove that the zero solution is globally asymptotic stable and they deduced the rate of decay of the solution when $t\rightarrow\infty$. Extension of these results to the multidimensional case was established in \cite{gur2010}. Continuous dependence of the solution on the parameters of the problem was shown in \cite{celebi}. For more results on the decay and asymptotic behaviour of solutions of semilinear hyperbolic problems and nonlinear dissipative wave equations refer to \cite{haraux88,nakao91} and the references within. Also more details regarding the continuous dependence on the parameters can be found in \cite{ames}.\\

In this article we consider the same model as in \cite{celebi}; only we assume the coefficient of the effective tension to be a function in the spatial variables. Unlike previous studies, we assume the coefficient of the nonlinear drag force to be a function in time and we also consider a time dependent boundary conditions at the top end of the riser. This article will be organized such that we state our problem in section \ref{statement}. In section \ref{pre} we introduce some essential estimates to be used repeatedly. Section \ref{estimates} contains an introduction to a form of an Energy functional and we investigate its properties under some conditions on the paramters and the boundary conditions. In section \ref{convergence} we estimate the rate of decay of this functional, from which we deduce the allowed rate of growth in time of the boundary conditions and the coefficient of the nonlinear drag force so that the structural stability of the system remains maintained globally in time. Finally, in section \ref{con} we summarize our conclusions.
%%%%%%%%%%%%%%%%%%%%%%%%%%%%%%%%%%%%%%%%%
\section{Statement of the problem}
\label{statement}
In this article we will try to represent a practical study of the marine riser problem by trying to imitate some realistic scenario. On one hand we will restrict ourselves to a cylindrical domain in $\mathbb{R}^3$. That is to say that $\Omega$ is a uniform cylinder of radius $\rho$ and height $h$ placed vertically such that $z\in[0,h]$, where $z=0$ represents the bottom of the riser at seabed and $z=h$ represents the top end of the riser. On another hand we will define some of the parameters of the problem as functions in either the spatial or the time variables. Our model takes the form:
\begin{equation}
\left.
 \begin{IEEEeqnarraybox}[
 \IEEEeqnarraystrutmode
 \IEEEeqnarraystrutsizeadd{5pt}
 {2pt}
 ][c]{l}
 u_{tt}+k\Delta^2u-\nabla\cdot\left(a(x)\nabla u\right) +\vec{g}\cdot\nabla u_t+b(t)|u_t|^pu_t=0,\,\text{in}\;\Omega\times(0,t),
 \\
 u(x,0)=u_0(x),\quad u_t(x,0)=u_1(x),\quad\text{in}\;\Omega,
 \\
 \left.u\right|_{z=h}=\phi(t),\,\left.\frac{\partial u}{\partial n}\right|_{z=h}=\alpha(t),\,\text{and}\,\left.u\right|_{z=0}=\left.\frac{\partial u}{\partial n}\right|_{z=0}=\left.\Delta u\right|_{\partial\Omega}=0.
 \end{IEEEeqnarraybox}
\right.
\label{mainmodel}
\end{equation}

This model represents a marine riser with the top end being allowed to move horizontally, where $\phi(t)$ and $\alpha(t)$ can be thought of as the dynamics of the top end of the riser due to the effects of waves, currents or any other form of external forces. The solution $u(x,t)$ represents the deflection of the pipe, and since we chose a cylindrical domain then it is natural to assume that the deflection is radial; that is to say that $u(x,t)\geq0$. Being the case, we shall assume as well that $\phi(t)\geq0$ is the radial (horizontal) motion of the top end. However, we shall not make any similar assumptions about the sign of $\alpha(t)$.\\

In the above model we take $p\geq 1$ as a given number. The constant $k>0$ denotes the flexural rigidity of the pipe. We will consider the coefficient of the effective tension as a function in the space such that $a:\mathbb{R}^3\mapsto\mathbb{R}$ and we assume that $a(x)$ is a $C^0$ function. We denote by $a_h$ the value of $a(x)$ at $z=h$. The coefficient of the nonlinear drag force $b:\mathbb{R}\mapsto\mathbb{R}$ is a $C^0$ function and such that $b(t)\geq b_{0}>0$.\\

The effect of the Coriolis force is a constant vector such that $\vec{g}=\{g_1,g_2,g_3\}$ and for simplicity we shall assume that $g_3>0$. The initial functions $u_0(x)$ and $u_1(x)$ are chosen such that $u_0\in H^2(\Omega)$ and $u_1\in L^2(\Omega)$. Finally, the function $\phi(t)\in C^1[0,\infty)$ and $\alpha(t)\in C^0[0,\infty)$.
%%%%%%%%%%%%%%%%%%%%%%%%%%%%%%%%%%%%%%%%%%%%%%%%
\section{Preliminaries:}
\label{pre}
In this section we derive some inequalities that we shall use repeatedly through the rest of this article, and we introduce each of them in a form of a corollary. In the rest of this article we will use the notation $\|.\|$ to denote the norm in the space of $L^2(\Omega)$. Otherwise, we shall use the standard sub-indexed terminology. We will also make the assumption that $h>1/2$; it is logical to make such an assumption given the nature of the problem. 
%%%%%%%%%%%%%%%%%%%%%%%%%%%%%%%%%%5555
\begin{cor}
Given the boundary conditions in \eqref{mainmodel}, the following estimate holds:
\begin{equation}
\int_\Omega u^2dx\leq 4h^2\left(\int_\Omega|\nabla u|^2dx+\pi\rho^2\phi^2(t)\right).
\label{friedineq}
\end{equation}
\label{Friedest}
\end{cor}
\begin{proof}
Consider the boundary conditions in \eqref{mainmodel}, use Cauchy's inequality and magnify the boundary condition by a factor of $2h$ as follows
\begin{IEEEeqnarray}{rCl}
u^2(\rho,\varphi,z,t)&=&\phi^2(t)-2\int_z^h u(\rho,\varphi,s,t)u_z(\rho,\varphi,s,t)ds\nonumber\\
&\leq&2\int_0^h|uu_z|dz+2h\phi^2(t)\nonumber\\
&\leq&\frac{1}{\varepsilon}\int_0^hu^2dz+\varepsilon\int_0^h|u_z|^2ds+2h\phi^2(t)\nonumber\\
&\leq&\frac{1}{2h}\int_0^hu^2dz+2h\int_0^h|\nabla u|^2ds+2h\phi^2(t)\nonumber
%\label{}
\end{IEEEeqnarray}
By integrating over $\Omega$ and applying {\it Fubini's Theorem} we get our result. 
\end{proof}
%%%%%%%%%%%%%%%%%%%%%%%%%%%%%%%%%%%%%%%%%%%%%%5
\begin{cor}
Given the boundary conditions in \eqref{mainmodel}, the following estimate holds:
\begin{IEEEeqnarray}{rCl}
\int_\Omega|\nabla u|^2dx&\leq&q(t)+4h^2\int_\Omega|\Delta u|^2dx,
\label{gradineq}
\end{IEEEeqnarray}
where $q(t)=\pi\rho^2(2\phi\alpha+\phi^2)$.
\label{Gradest}
\end{cor}
\begin{proof}
Apply the {\it Divergence Theorem}, consider the boundary conditions in \eqref{mainmodel} and use Cauchy's inequality as follows
\begin{IEEEeqnarray}{rCl}
\int_\Omega|\nabla u|^2dx&=&\int_\Omega\nabla\cdot\left(u\nabla u\right)dx-\int_\Omega u\Delta udx\nonumber\\
&=&\int_{\partial\Omega}u\frac{\partial u}{\partial n}ds-\int_\Omega u\Delta udx\nonumber\\
&\leq&\pi\rho^2\phi\alpha+\frac{\varepsilon}{2}\int_\Omega u^2dx+\frac{1}{2\varepsilon}\int_\Omega|\Delta u|^2dx\nonumber\\
&\leq&\pi\rho^2\phi\alpha+2\varepsilon \pi\rho^2h^2\phi^2+2\varepsilon h^2\int_\Omega |\nabla u|^2dx+\frac{1}{2\varepsilon}\int_\Omega|\Delta u|^2dx,\nonumber
\end{IEEEeqnarray}
where we used \eqref{friedineq} in the last inequality and by choosing $\varepsilon=1/4h^2$ we obtain our result.
\end{proof}
%%%%%%%%%%%%%%%%%%%%%%%%%%%%%%%%%%%%%%%%%%%%%%%%%
\begin{cor}
Given the boundary conditions in \eqref{mainmodel}, the following estimate holds:
\begin{IEEEeqnarray}{rCl}
\frac{\pi\rho^2\alpha^2}{h}\leq\int_\Omega|\Delta u|^2dx.
\label{deltaineq}
\end{IEEEeqnarray}
\label{Deltaest}
\end{cor}
\begin{proof}
We employ the {\it Divergence Theorem} and H\"older's inequality as follows
\begin{IEEEeqnarray}{rCl}
\pi\rho^2\alpha\leq\pi\rho^2|\alpha|=\left|\int_{\partial\Omega}\nabla u\cdot \vec{n}\,ds\right|=\left|\int_\Omega \Delta udx\right|\leq\left(\int_\Omega 1^2dx\right)^{\frac{1}{2}}\left(\int_\Omega |\Delta u|^2dx\right)^{\frac{1}{2}},\nonumber
\end{IEEEeqnarray}
and squaring concludes the proof.
\end{proof}
%%%%%%%%%%%%%%%%%%%%%%%%%%%%%%%%%%%%%%%%%%%%%%
\begin{rem}
The estimates of {\it Corollaries \ref{Friedest} and \ref{Gradest}} can be enhanced significantly as shown in \cite{rekt}. However, as noted in {\it Corollary 1}, we magnified the estimate for the purpose of simplification. Yet, this shall not affect our conclusions. 
\end{rem}
%%%%%%%%%%%%%%%%%%%%%%%%%%%%%%%%%%%%%%%%%%%%%%%%5
\section{The Energy functional $E(t)$:}
\label{estimates}
In order to study the qualitative properties of the solution we need first to assume the existence of such solution. So, we assume the existence of a weak solution for \eqref{mainmodel} defined as follows.
\begin{defn}
A function $u(x,t)$ is called a weak solution of \eqref{mainmodel} if $u\in C([0,t],H^2(\Omega))$ and $u_t\in C([0,t],L^2(\Omega))\cap L^{p+2}([0,t],L^{p+2}(\Omega))$ for all $t>0$, and it satisfies \eqref{mainmodel} in the sense of distributions.  
\label{weaksol}
\end{defn}

We now define the energy functional $E(t)$ in the form:
\begin{equation}
E(t)=\frac{1}{2}\left(\|u_t\|^2+k\|\Delta u\|^2+\int_\Omega a(x)|\nabla u|^2dx\right),
\label{energyf}
\end{equation}
and we study its bounds in light of {\it Corollaries \ref{Gradest} and \ref{Deltaest}}.
%%%%%%%%%%%%%%%%%%%%%%%%%%%%%%%%%%%%%%%%%%%%%%%%%%%5
\begin{lem}
Assume there exists a weak solution for \eqref{mainmodel} and that $k>4\max|a(x)|h^2$. If for some $\delta>1$ the functions $\phi$ and $\alpha$ are satisfying the condition:
\begin{equation}
\frac{(2|\alpha|\phi+\phi^2)}{\alpha^2}\leq\frac{k_0}{\delta\widehat{a}h}
\label{condK}
\end{equation}
where $\widehat{a}=\max|a(x)|$ and $k_0=k-4\widehat{a}h^2$, then the functional $E(t)$ is bounded such that:
\begin{equation}
\|u_t\|^2+k_0\left(\frac{\delta-1}{\delta}\right)\|\Delta u\|^2\leq 2E(t)\leq \|u_t\|^2+\left(\frac{(\delta+1)k+(\delta-1)4\widehat{a}h^2}{\delta}\right)\|\Delta u\|^2
\label{condE}
\end{equation}
 for all $t>0$. 
\label{boundE}
\end{lem}
\begin{proof}
If $u(x,t)$ is the weak solution of \eqref{mainmodel} and given that $a(x)\in C^0(\Omega)$ then $E(t)$ is certainly bounded from above and its bound can be estimated by virtue of estimate \eqref{gradineq}, condition \eqref{condK} and estimate \eqref{deltaineq} as follows
\begin{IEEEeqnarray}{rCl}
E(t)&\leq&\frac{1}{2}\left(\|u_t\|^2+k\|\Delta u\|^2+\widehat{a}\|\nabla u\|^2\right)\nonumber\\
&\leq&\frac{1}{2}\left(\|u_t\|^2+k\|\Delta u\|^2+\widehat{a}\left(\pi\rho^2(2\phi|\alpha|+\phi^2)+4h^2\|\Delta u\|^2\right)\right)\nonumber\\
&\leq&\frac{1}{2}\left(\|u_t\|^2+k\|\Delta u\|^2+\frac{\pi\rho^2\alpha^2k_0}{\delta h}+4\widehat{a}h^2\|\Delta u\|^2\right)\nonumber\\
&\leq&\frac{1}{2}\left(\|u_t\|^2+k\|\Delta u\|^2+\frac{k_0}{\delta}\|\Delta u\|^2+4\widehat{a}h^2\|\Delta u\|^2\right),\nonumber
%\label{condE1}
\end{IEEEeqnarray}
which yields the upper bound. The lower bound can be estimated in a similar fashion and this completes the proof.
\end{proof}
%%%%%%%%%%%%%%%%%%%%%%%%%%%%%%%%%%%%%%%%%%%%%%%%%5
\begin{rem}
The condition \eqref{condK} is more important to derive the lower bound than how it is to the upper one. It depends mainly on how we choose $\alpha(t)$. If we choose $\alpha<0$ and such that $2\alpha\phi+\phi^2<0$ for all $t>0$, then we can neglect this term when estimating the lower bound and we obtain $2E\geq \|u_t\|^2+k_0\|\Delta u\|^2$, which is the bound commonly introduced in the literature. We are proposing this condition to allow a flexible choice for $\alpha$. 
\end{rem}
%%%%%%%%%%%%%%%%%%%%%%%%%%%%%%%%%%%%%%%%%%%%%
Let us now define the function $d(t)$ such that $d_t(t)$ is defined as
\begin{equation}
d_t(t)=\pi\rho^2\left(a_h\phi_t\alpha-\frac{g_3}{2}\phi_t^2\right),
\label{negterm}
\end{equation}
so that we can state the next lemma.
%%%%%%%%%%%%%%%%%%%%%%%%%%%%%%%%%%%%%%%%%%%%%%
\begin{lem}
Assume the conditions of {\it Lemma \ref{boundE}} are satisfied. If the function $d(t)$ satisfies the following condition:
\begin{equation}
d_t(t)<0,
\label{condD}
\end{equation}
for all $t>0$. Then $E(t)$ is a Lyapunov functional and the following estimates hold:
\begin{equation}
E(t)\leq E(0),
\label{estE}
\end{equation}
and
\begin{equation}
\int_0^t\int_\Omega|u_t|^{p+2}dxdt\leq \frac{E(0)}{b_0},
\label{estIb}
\end{equation}
for all $t>0$. Consequently, the zero solution is stable.
\label{lyapunovE}
\end{lem}
\begin{proof}
Multiply \eqref{mainmodel} by $u_t$, integrate by parts over $\Omega$, apply the {\it Divergence Theorem} and employ the predefined boundary conditions to get
\begin{IEEEeqnarray}{rCl}
\frac{d}{dt}E(t)+I_b(x,t)&=&d_t(t).
\label{step1}
\end{IEEEeqnarray}
where 
\begin{equation}
I_b(x,t)=\int_\Omega b(t)|u_t|^{p+2}dx,
\label{dragint}
\end{equation}
It is clear that $I_b>0$, and by condition \eqref{condD} we have
\begin{equation}
\frac{d}{dt}E(t)<0.
\label{negrate}
\end{equation}
We integrate \eqref{step1} to get
\[
E(t)+\int_0^tI_b(x,s)ds=E(0)+d(t)-d(0).
\]
By condition \eqref{condD} we have $d(t)-d(0)<0$ so that
\[
E(t)+\int_0^tI_b(x,s)ds\leq E(0).
\]
By our assumption that $0<b_0\leq b(t)$ we conclude estimates \eqref{estE} and \eqref{estIb}. By {\it Lemma \ref{boundE}} and inequality \eqref{negrate} we conclude that $E(t)$ is a Lyapunov functional, which in turn implies the stability of the zero solution.
\end{proof}
%%%%%%%%%%%%%%%%%%%%%%%%%%%%%%%%%%%%%%%%%%%%%%%
\begin{rem}
We need to highlight that condition \eqref{condD} does not imply that $\phi$ is a decreasing function nor that $\alpha$ is a non-positive one. We can have $\phi$ increasing and $\alpha$ non-negative, and even our assumption that $g_3>0$ can be dropped if, for example, we have $a_h<-C$ for some sufficiently large constant $C>0$. However, our choice to assume that $g_3>0$ was due to our understanding that this might be the most controllable parameter in the above condition. 
\end{rem}
%%%%%%%%%%%%%%%%%%%%%%%%%%%%%%%%%%%%%%%%%%%%
\section{Global asymptotic stability:}
\label{convergence}
In this part we try to deduce the rate of decay of $E(t)$ in the large so that we can confirm that the zero solution is globally asymptotic stable. Such result should lead us to an estimate on the maximum rate of growth in time for $\phi(t)$, $\alpha(t)$ and $b(t)$. To this end we state our main theorem. 
\begin{thm}
Assume the conditions of {\it Lemma \ref{lyapunovE}} are satisfied. If the functions $\phi(t)$, $\alpha(t)$ and $b(t)$ are satisfying the condition:
\begin{equation}
\phi(t)< M_1t^{m},\,\alpha(t)< M_2t^{n}\quad\text{and}\quad b(t)\leq M_3t^{\frac{p+1-\iota}{p+2}}\quad\text{for}\quad t\gg1,
\label{rategrow}
\end{equation}
where $m<1/2$, $n<-m$, $M_1$, $M_2$ and $M_3$ are constants that do not depend on $u$ nor $t$. Then the functional $E(t)$ admits a rate of decay:
\begin{equation}
E(t)< M\max\left(t^{m+n},t^{2m-1},t^{\frac{-\iota}{p+2}},t^{\frac{-2}{p+1}}\right)\quad\text{for}\quad t\gg1,
\label{ratedecay}
\end{equation} 
where $M$ is a constant that does not depend on $u$ nor $t$. In this case, the zero solution is globally asymptotic stable. 
\label{mainthm}
\end{thm}
\begin{proof}
Multiply \eqref{mainmodel} by $u$, integrate by parts, apply the {\it Divergence Theorem} and employ the boundary conditions to get
\begin{IEEEeqnarray}{rCl}
\frac{d}{dt}\int_\Omega uu_tdx+k\int_\Omega|\Delta u|^2dx+\int_\Omega a(x)|\nabla u|^2dx&=&r(t)+I_1-I_2+I_3,
\label{step2}
\end{IEEEeqnarray}
where $r(t)=\pi\rho^2(a_h\alpha\phi-g_3\phi\phi_t)$ and
\[
I_1=\int_\Omega g\cdot\nabla u u_tdx,\quad I_2=\int_\Omega b(t)|u_t|^{p}u_tudx\quad\text{and}\quad I_3=\int_\Omega|u_t|^2dx.
\]
Define the functional $H(t)$ as 
\begin{equation}
H(t)=\int_\Omega uu_tdx+\sigma E(t).
\label{funcH}
\end{equation}   
Thus, by recalling the definition of $E(t)$, we write equation \eqref{step2} in the form
\begin{IEEEeqnarray}{rCl}
\frac{d}{dt}H-\sigma\frac{d}{dt}E+2E&=&r(t)+I_1-I_2+2I_3,\nonumber
%\label{condH1}
\end{IEEEeqnarray}
and using \eqref{step1} yields
\begin{IEEEeqnarray}{rCl}
\frac{d}{dt}H-\sigma d_t+\sigma I_b+2E&=&r(t)+I_1-I_2+2I_3,\nonumber
%\label{condH1}
\end{IEEEeqnarray}
which when rearranged produces
\begin{IEEEeqnarray}{rCl}
2E-\sigma d_t&=&-\frac{d}{dt}H+r(t)+I_1-I_2+2I_3-\sigma I_b,
\label{condH1}
\end{IEEEeqnarray}
Let us analyse the integral $I_1$ using Cauchy's inequality and estimate \eqref{gradineq} as follows
\begin{IEEEeqnarray}{rCl}
\int_\Omega g\cdot\nabla uu_tdx&\leq&|g|\int_\Omega|\nabla u||u_t|dx\nonumber\\
&\leq&\frac{|g|}{2}\left(\varepsilon\int_\Omega|\nabla u|^2dx+\varepsilon^{-1}\int_\Omega|u_t|^2dx\right)\nonumber\\
&\leq&\frac{|g|}{2}\left(\varepsilon\pi\rho^2(2|\alpha|\phi+\phi^2)+4\varepsilon h^2\|\Delta u\|^2+\varepsilon^{-1}\int_\Omega|u_t|^2dx\right).\nonumber
%\label{I1}
\end{IEEEeqnarray}
We use estimate \eqref{deltaineq} and condition \eqref{condK} to reach
\begin{IEEEeqnarray}{rCl}
\int_\Omega g\cdot\nabla uu_tdx&\leq&\frac{|g|}{2}\left(\frac{\varepsilon k_0}{\delta\widehat{a}}\|\Delta u\|^2+4\varepsilon h^2\|\Delta u\|^2+\varepsilon^{-1}\int_\Omega|u_t|^2dx\right)\nonumber\\
&\leq&\frac{|g|}{2}\left(\frac{\varepsilon (k+4\widehat{a}h^2)}{2\widehat{a}}\|\Delta u\|^2+\varepsilon^{-1}\int_\Omega|u_t|^2dx\right),\nonumber
%\label{I1}
\end{IEEEeqnarray}
where we used $\delta=2$. By estimate \eqref{condE} we have $4E\geq k_0\|\Delta u\|^2$, and by using this estimate in the last inequality and setting $\varepsilon=\widehat{a}k_0/(|g|(k+4\widehat{a}h^2))$ we obtain
\begin{IEEEeqnarray}{rCl}
\int_\Omega g\cdot\nabla uu_tdx&\leq&E+\frac{|g|^2(k+4\widehat{a}h^2)}{2\widehat{a}k_0}\int_\Omega|u_t|^2dx.
\label{I1}
\end{IEEEeqnarray}
Therefore, recalling that $0<b_0\leq b(t)$, $d_t<0$ and using \eqref{I1} in equation \eqref{condH1} transforms it into the following inequality
\begin{IEEEeqnarray}{rCl}
E&\leq&-\frac{d}{dt}H+r(t)+\hat{I}_2+B I_3-b_0\sigma \hat{I}_b,
\label{condH2}
\end{IEEEeqnarray}
where $B=2+|g|^2(k+4\widehat{a}h^2)/(2\widehat{a}k_0)$ and
\[
\hat{I_2}=\int_\Omega b(t)|u_t|^{p+1}|u|dx\quad\text{and}\quad \hat {I_b}=\int_\Omega|u_t|^{p+2}dx.
\]

Now we check the lower bound of $H$ by using Cauchy's inequality, estimates \eqref{friedineq} and \eqref{gradineq} as follows
\begin{IEEEeqnarray}{rCl}
H=\int_\Omega uu_tdx+\sigma E&\geq&-\frac{1}{2\varepsilon}\int_\Omega u^2dx-\frac{\varepsilon}{2}\int_\Omega u_t^2dx+\sigma E\nonumber\\
&\geq&-\frac{2\pi\rho^2h^2}{\varepsilon}\phi^2-\frac{2h^2}{\varepsilon}\int_\Omega|\nabla u|^2dx-\frac{\varepsilon}{2}\int_\Omega u_t^2dx+\sigma E\nonumber\\
&\geq&-\frac{2h^2}{\varepsilon}\left(q(t)+\pi\rho^2\phi^2+4h^2\|\Delta u\|^2\right)-\frac{\varepsilon}{2}\|u_t\|^2+\sigma E.\nonumber\\
&\geq&-\frac{2h^2}{\varepsilon}\left(2\pi\rho^2(\phi|\alpha|+\phi^2)+4h^2\|\Delta u\|^2\right)-\frac{\varepsilon}{2}\|u_t\|^2+\sigma E.\nonumber
%\label{condH4}
\end{IEEEeqnarray}
By condition \eqref{condK} and estimate \eqref{deltaineq} we reach
\begin{IEEEeqnarray}{rCl}
H&\geq&-\frac{2h^2}{\varepsilon}\left(\frac{2k_0}{\delta\widehat{a}}\|\Delta u\|^2+4h^2\|\Delta u\|^2\right)-\frac{\varepsilon}{2}\|u_t\|^2+\sigma E\nonumber\\
&\geq&-\frac{2h^2k_0}{\varepsilon\widehat{a}}\|\Delta u\|^2-\frac{\varepsilon}{2}\|u_t\|^2+\sigma E\nonumber\\
&\geq&-\frac{2hk_0}{\sqrt{\widehat{a}}}\|\Delta u\|^2-\frac{h}{2\sqrt{\widehat{a}}}\|u_t\|^2+\sigma E\nonumber\\
&\geq&-\frac{2h}{\sqrt{\widehat{a}}}\left(k_0\|\Delta u\|^2+2\|u_t\|^2\right)+\sigma E,\nonumber
%\label{condH4}
\end{IEEEeqnarray}  
where we used $\delta=2$, $\varepsilon=h/\sqrt{\widehat{a}}$ and we minimized the second term by a factor of $4$. Recalling the bound in \eqref{condE} we obtain
\begin{IEEEeqnarray}{rCl}
H&\geq&\left(\sigma-\frac{8h}{\sqrt{\widehat{a}}}\right) E,
\label{boundH}
\end{IEEEeqnarray}
such that for any $\sigma>8h/\sqrt{\widehat{a}}$ we have $H\geq\mu E\geq0$.\\

We now return to \eqref{condH2} and we integrate over the interval $(0,t)$ to get
\begin{IEEEeqnarray}{rCl}
\int_0^tE(s)ds&\leq&H(0)-H(t)+\int_0^tr(s)ds+\int_0^t(\hat{I_2}+BI_3-b_0\sigma \hat{I_b})ds.
\label{condH3}
\end{IEEEeqnarray}
Since $H(t)\geq0$ for all $t>0$ then $H(0)-H(t)\leq H(0)$. We know that $\hat{I_b}>0$, and also $E(t)$ is a non-increasing function such that
\[
tE(t)\leq \int_0^tE(s)ds.
\]
Therefore, equation \eqref{condH3} takes the form
\begin{IEEEeqnarray}{rCl}
tE(t)&\leq&H(0)+\int_0^tr(s)ds+\int_0^t(\hat{I_2}+BI_3+b_0\sigma I_b)ds.
\label{condH4}
\end{IEEEeqnarray}
To estimate the above integrals we proceed as follows
\begin{IEEEeqnarray}{rCl}
b_0\sigma\int_0^tI_bds&\leq&\sigma E(0)=C_b,
\label{Ib}
\end{IEEEeqnarray}
as per \eqref{estIb}. Next, we check the integral $I_3$ using H\"older's inequality and estimate \eqref{estIb} as follows
\begin{IEEEeqnarray}{rCl}
B\int_0^t\int_\Omega|u_t|^2dxds&\leq&B\left(\int_0^t\int_\Omega1dxds\right)^{\frac{p}{p+2}}\left(\int_0^t\int_\Omega|u_t|^{p+2}dxds\right)^{\frac{2}{p+2}}\nonumber\\
&\leq&B(\pi\rho^2h)^{\frac{p}{p+2}}\left(\frac{E(0)}{b_0}\right)^{\frac{2}{p+2}}t^{\frac{p}{p+2}}\nonumber\\
&\leq&C_3t^{\frac{p}{p+2}}.
\label{I3}
\end{IEEEeqnarray}
It remains to evaluate the integral $\hat{I_2}$. To this end we need to recall estimates \eqref{condE} and \eqref{estE} which imply that $\|\Delta u\|^2\leq 4E/k_0\leq 4E(0)/k_0$ assuming $\delta=2$. The boundedness of $u$ and $\nabla u$ in $L^2(\Omega)$ follows by estimates \eqref{friedineq} and \eqref{gradineq}. It follows that $u\in H^2(\Omega)$. In particular, by {\it Sobolev Embedding Theorem}, we deduce that $u\in C^{0,\frac{1}{2}}(\overline{\Omega})$ and such that
\begin{equation}
\|u\|_{C^{0,\frac{1}{2}}(\overline{\Omega})}\leq\|u\|_{H^2(\Omega)}\leq\|\Delta u\|_{L^2(\Omega)}\leq 2\sqrt{\frac{E(0)}{k_0}}.
\label{holderest}
\end{equation}
Assume $b(t)$ is growing in time such that $b(t)<M_3t^\lambda$. Bearing that in mind we proceed with $\hat{I_2}$ by employing H\"older's inequality, estimate \eqref{estIb} and estimate \eqref{holderest} as follows
\begin{IEEEeqnarray}{rCl}
\int_0^t\int_\Omega b(t)|u||u_t|^{p+1}dxds&\leq&\left(\int_0^t\int_\Omega |u_t|^{p+2}dxds\right)^{\frac{p+1}{p+2}}\times\nonumber\\
&&\left(\int_0^t\int_\Omega \left(b(t)|u|\right)^{p+2}dxds\right)^{\frac{1}{p+2}}\nonumber\\
&\leq& \left(\frac{E(0)}{b_0}\right)^{\frac{p+1}{p+2}}\left(\frac{4E(0)}{k_0}\right)^{\frac{1}{2}}M_3\left(\int_0^tt^{\lambda(p+2)}ds\right)^{\frac{1}{p+2}}\nonumber\\
&\leq&C_2t^{\frac{\lambda (p+2)+1}{p+2}}.
\label{I2}
\end{IEEEeqnarray}
We rewrite \eqref{condH4} in light of the above estimates
\begin{equation}
E(t) \leq t^{-1}H(0)+t^{-1}\int_0^tr(s)ds+C_2t^{\frac{p(\lambda-1)+2(\lambda-1)+1}{p+2}}+C_3t^{\frac{-2}{p+2}}+C_bt^{-1}.
\label{condH5}
\end{equation}
Clearly the first and the last terms in the right hand side are the fastest terms to decay and so they can be ignored. Finally, let us analyse the two terms of $r(t)$ separately. Assume $\phi(t)$ is growing such that $\phi<M_1t^m$. The term $\pi\rho^2g_3\phi\phi_t$ can be integrated as follows
\[
t^{-1}\int_0^t\phi\phi_tdt=\frac{1}{2}t^{-1}(\phi^2(t)-\phi^2(0))< \frac{M_1^2}{2}t^{2m-1}+Ct^{-1}.
\]
As for the second term, assume the worst scenario; that is $\alpha$ also is growing in time such that $\alpha<t^n$ then we can write
\[
t^{-1}\int_0^t\alpha\phi dt\leq\alpha\phi<M_1M_2t^{m+n}.
\]
Accordingly, by the last two inequalities and \eqref{condH5} we conclude that
\begin{IEEEeqnarray}{rCl}
E(t)&<&M\max_{t\gg1}\left(t^{m+n},t^{2m-1},t^{\frac{p(\lambda-1)+2(\lambda-1)+1}{p+2}},t^{\frac{-2}{p+2}}\right),
\label{condH6}
\end{IEEEeqnarray}
where $M$ is a constant that does not depend on $t$ nor the solution $u$. From this estimate we conclude that if $m<1/2$, $n<-m$ and $\lambda\leq(p+1-\iota)/(p+2)$ then the zero solution is globally asymptotic stable.
\end{proof}
%%%%%%%%%%%%%%%%%%%%%%%%%%%%%%%%%%%%%%%%%%
\begin{rem}
We need to highlight that the statement of the last theorem is not actually decisive when it comes to determining consistent constraints on the boundary conditions $\phi$ and $\alpha$. To understand this point we need to recall condition \eqref{condK}. This condition was essential to derive our estimates and to validate our conclusions. However, and without loss of generality, if we express $\phi$ and $\alpha$ in terms of polynomials in time such that $\phi\sim t^m$ and $\alpha\sim t^n$ then the left hand side of condition \eqref{condK} is a polynomial in the form $t^{m-n}+t^{2(m-n)}$. The right hand side of this condition is a constant. Therefore, it is necessary that $m-n<0$, but $n<-m$ as we concluded above. It follows that it is necessary that $m<0$. This simply says that the pipe can return to its vertical position as long as the horizontal effects on the top end are decreasing with time. Intuitively, this makes a great sense in the light of our assumptions which were all focusing solely on the restrictions on $\phi$ and $\alpha$ without linking them directly to any other parameters. This last statement should be a subject for a further study.
\end{rem}
%%%%%%%%%%%%%%%%%%%%%%%%%%%%%%%%%%%%%%%%%%%%%%%%%%%%
\section{Conclusion:}
\label{con}
We have shown that the solution of the initial-boundary value problem for the marine riser equation converges to zero and such that the zero solution is globally asymptotic stable in the presence of time dependent boundary conditions at the top end and time dependent coefficient of the nonlinear drag force given certain conditions on these functions and their permitted rates of growth.
%%%%%%%%%%%%%%%%%%%%%%%%%%%%%%%%%%%%%%%%%
%\bibliographystyle{spmpsci}
\def\cprime{$'$} \def\cprime{$'$}
 ------------------------------------------------------------------------
\end{document}